\documentclass[11pt]{article}

\usepackage{geometry}
\usepackage{amsthm}
\usepackage{amssymb}
\usepackage{amsfonts}
\usepackage{amsmath}
\usepackage{mathrsfs}
\usepackage{bbm}      
\usepackage[sort&compress]{natbib}

\usepackage{graphicx} 
\usepackage[all]{xy} 

\newtheorem{theorem}{Theorem}

\newcommand{\field}[1]{\mathbb{#1}}
\newcommand{\R}{\field{R}}
\newcommand{\E}{\field{E}}
\newcommand{\EXP}{\E}

\newcommand{\IND}[1]{\mathbbm{1}_{\{ #1 \}}}

\newcommand{\norm}[1]{\left\|{#1}\right\|}
\newcommand{\argmin}{\mathop{\rm argmin}}

\newcommand{\B}{\mathcal{B}}

\newcommand{\X}{\mathcal{X}}

\newcommand{\Pa}{\mathcal{P}}

\newcommand{\ceil}[1]{\left\lceil #1\right\rceil }

\newcommand{\example}{\medskip\noindent{\bf Example.} \ }

\newtheorem{prop}[theorem]{Proposition}
\newtheorem{thm}[theorem]{Theorem}
\newtheorem{cor}[theorem]{Corollary}

\newcommand{\prob}[1]{\mathbb{P}\left\{#1\right\}} 
\newcommand{\probi}[2]{\mathbb{P}_{#1}\left\{#2\right\}} 
\newcommand{\esp}[1]{\mathbb{E}\left[#1\right]} 
\newcommand{\espi}[2]{\mathbb{E}_{#1}\left[#2\right]} 
\newcommand{\var}[1]{\text{Var}\left(#1\right)} 

\title{Robust estimation of U-statistics}
\author{Emilien Joly
  \thanks{\'Ecole Normale Sup\'erieure, Paris}
\and
G\'abor Lugosi\thanks{ICREA and Department of Economics and Business, Pompeu Fabra University. Supported by the Spanish Ministry of Science and Technology grant MTM2012-37195.
}}

\begin{document}
\maketitle

\begin{center}
{\sl This paper is dedicated to the memory of Evarist Gin\'e.}
\end{center}

\begin{abstract}
An important part of the legacy of Evarist Gin\'e is his fundamental contributions
to our understanding of $U$-statistics and $U$-processes.
In this paper we discuss the estimation of the mean of multivariate functions in case of possibly heavy-tailed distributions. In such situations, reliable estimates of the mean cannot be obtained by usual $U$-statistics. We introduce a new estimator, based on the so-called median-of-means technique. We develop performance bounds for this new estimator that
generalizes an estimate of \cite{ArGi93}, showing that the new estimator performs, under minimal moment conditions, as well as classical $U$-statistics for bounded random variables. We discuss an application of this estimator to clustering. 
\end{abstract}

\section{Introduction}
\label{sec:intro}
Motivated by numerous applications, the theory of $U$-statistics and $U$-processes has received considerable attention in the past decades. $U$-statistics appear naturally in \emph{ranking} \citep{CleVaLu08}, \emph{clustering} \citep{Cle14} and \emph{learning on graphs} \citep{biau06} or as components of higher-order terms in expansions of smooth statistics, see, for example, \cite{ro09}.
The general setting may be described as follows. Let $X$ be a random variable taking values in some measurable space $\X$ and let $h : \X^m \rightarrow \R$ be a measurable function of $m\ge 2$ variables. Let $P$ be the probability measure of $X$. Suppose we have access to $n\ge m$ independent random variables $X_1,\ldots,X_n$, all distributed as $X$. We define the $U$-statistics of order $m$ and kernel $h$ based on the sequence $\{X_i\}$ as
\begin{equation}
\label{eq:defUstat}
U_n(h)=\frac{(n-m)!}{n!}\sum_{(i_1,\ldots,i_m) \in I_n^m} h(X_{i_1},\ldots,X_{i_m})~, 
\end{equation}
where
\[
I_n^m=\left\{(i_1,\ldots,i_m) : 1 \le i_j \le n,\ i_j \neq i_k \ \ \text{if}\ \ j\neq k \right\}
\]
is the set of all $m$-tuples of different integers between $1$ and $n$.
$U$-statistics  are unbiased estimators of the mean $m_h=\EXP h(X_1,\ldots,X_m)$ and have minimal variance among all unbiased estimators \citep{Ho48}. 
Understanding the concentration of a $U$-statistics around its expected value has been subject of extensive study. \cite{PeGi99} provide an excellent summary
but see also \cite{GiLaZi01} for a more recent development.

By a classical inequality of \cite{Hoe63},
for a bounded kernel $h$, for all $\delta >0$,
\begin{equation}
\label{eq:basicconcen}
\prob{|U_n(h)-m_h| > \norm{h}_{\infty}\sqrt{\frac{\log (\frac{2}{\delta})}{2\lfloor n/m \rfloor}}} \le \delta~,
\end{equation}
and we also have the ``Bernstein-type'' inequality
\[
\prob{|U_n(h)-m_h| > \sqrt{\frac{4\sigma^2\log (\frac{2}{\delta})}{2\lfloor n/m \rfloor}} \vee \frac{4\norm{h}_{\infty}\log (\frac{2}{\delta})}{6\lfloor n/m \rfloor}} \le \delta~,
\]
where $\sigma^2= \var{h(X_1,\ldots,X_m)}$.

However, under certain degeneracy assumptions on the kernel, significantly sharper bounds have been proved. Following the exposition of \cite{PeGi99}, 
for convenience, we restrict out attention to symmetric kernels.
A kernel $h$ is \emph{symmetric} if for all $x_1,\ldots,x_m \in \R$ and all permutations $s$,
\[
h(x_1,\ldots,x_m)=h(x_{s_1},\ldots,x_{s_m})~.
\]
A symmetric kernel $h$ is said to be \emph{$P$-degenerate} of order $q-1$, $1<q \le m$, if for all $x_1,\ldots,x_{q-1} \in \X$,
\[
\int h(x_1,\ldots,x_m)dP^{m-q+1}(x_q,\ldots,x_m) = \int h(x_1,\ldots,x_m)dP^m(x_1,\ldots,x_m) 
\]
and
\[
(x_1,\ldots,x_q) \mapsto \int f(x_1,\ldots,x_m)dP^{m-q}(x_{q+1},\ldots,x_m)
\]
is not a constant function. In the special case of $m_h=0$ and $q=m$ (i.e., when the kernel is $(m-1)$-degenerate, $h$ is said to be $P$-\emph{canonical}. 
$P$-canonical kernels appear naturally in the Hoeffding decomposition of a $U$-statistic, see  \cite{PeGi99}. 

\cite{ArGi93} proved the following important improvement of Hoeffing's 
inequalities for canonical kernels:
If $h-m_h$ is a bounded, symmetric $P$-canonical kernel of $m$ variables, there exist finite positive constants $c_1$ and $c_2$ depending only on $m$ such that
for all $\delta \in (0,1)$,
\begin{equation}
\label{eq:existresult}
\prob{|U_n(h)-m_h| \ge c_1\norm{h}_{\infty}\left(\frac{\log (\frac{c_2}{\delta})}{n}\right)^{m/2}} \le \delta~,
\end{equation}
and also
\begin{equation}
\label{eq:existresult2}
\prob{|U_n(h)-m_h| > \left(\frac{\sigma^2\log (\frac{c_1}{\delta})}{c_2 n}\right)^{m/2} \vee \frac{\norm{h}_{\infty}}{\sqrt{n}} \left(\frac{\log (\frac{c_1}{\delta})}{c_2}\right)^{(m+1)/2}} \le \delta~.
\end{equation}
In the special case of $P$-canonical kernels of order $m=2$, (\ref{eq:existresult}) implies that 
\begin{equation}
\label{eq:canonicalm2}
|U_n(h)-m_h| \le \frac{c_1\norm{h}_{\infty}}{n}\log \left(\frac{c_2}{\delta}\right)~,
\end{equation} 
with probability at least $1-\delta$. 
Note that this rate of convergence is significantly faster than the
rate $O_p(n^{-1/2})$ implied by
\eqref{eq:basicconcen}.

All the results cited above require boundedness of the kernel. 
If the kernel is unbounded but $h(X_1,\ldots,X_m)$ has sufficiently
light (e.g., sub-Gaussian) tails, then some of these results may be 
extended, see, for example, \cite{GiLaZi01}. However, if 
$h(X_1,\ldots,X_m)$ may have a heavy-tailed distribution, 
exponential inequalities do not hold anymore (even in the univariate $m=1$
case). However, even though $U$-statistics may have an erratic behavior
in the presence of heavy tails, in this paper we show that under minimal moment
conditions, one may construct estimators of $m_h$ that satisfy exponential
inequalities analogous to (\ref{eq:basicconcen}) and (\ref{eq:existresult}).
These are the main results of the paper. In particular, 
in Section \ref{sec:mainres} we introduce a robust estimator of the mean
$m_h$.
Theorems \ref{thm:maindege} and \ref{thm:momentp} establish exponential 
inequalities for the performance of the new estimator under minimal
moment assumptions. More precisely, Theorem \ref{thm:maindege} only requires
that $h(X_1,\ldots,X_m)$ has a finite variance and establishes inequalities
analogous to (\ref{eq:existresult}) for $P$-degenerate kernels.  
In Theorem \ref{thm:momentp} we further weaken the conditions
and only assume that there exists $1<p\le 2$ such that $\EXP |h|^p <\infty$.

The next example illustrates why classical $U$-statistics fail 
under heavy-tailed distributions.

\example{
Consider the special case $m=2$, $\EXP X_1 = 0$ and $ h(X_1,X_2)=X_1X_2 $. 
Note that this kernel is $P$-canonical.
We define $Y_1,\ldots,Y_n$ as independent copies of $X_1,\ldots,X_n$. By 
decoupling inequalities for the tail of $U$-statistics given in
Theorem 3.4.1 in \cite{PeGi99} (see also 
Theorem \ref{thm:decoutail} in the Appendix), $U_n(h)$ has a similar tail behavior to $\left(\frac{1}{n}\sum_{i=1}^nX_i\right)\left(\frac{1}{n-1}\sum_{j=1}^{n-1}Y_j\right) $. Thus, $U_n(h)$ behaves like a product of two independent empirical mean estimators of the same distribution. When the $X_i$ are heavy tailed, the empirical mean is known to be a poor estimator of the mean. As an example, assume that $X$ follows an $\alpha$-stable law $S(\gamma,\alpha)$ for some $\alpha \in (1,2)$ and $\gamma>0$.
Recall that a random variable $X$ has an $\alpha$-\emph{stable} law $S(\gamma,\alpha)$ if for all $ u \in \R$,
\[
\EXP \exp (iuX)=\exp(-\gamma^\alpha|u|^\alpha)
\]
(see \cite{Zo86}, \cite{No15}).
Then it follows from the properties of $\alpha$-stable distributions
(summarized in Proposition \ref{prop:stable} in the Appendix) that
there exists a constant $c>0$ depending only on $\alpha$ and $\gamma$ such that
\[
\prob{U_n(h) \ge n^{2/\alpha-2}} \ge c~,
\]
and therefore there is no hope to reproduce an upper bound like
(\ref{eq:canonicalm2}). Below we show how this problem can be dealt with
by replacing the $U$-statistics by a more robust estimator.
}

Our approach is based on robust mean estimators in the univariate setting.
Estimation of the mean of a possibly heavy-tailed random variable $X$ from  
i.i.d.\ sample $X_1,\ldots,X_n$
has recently received increasing attention. 
Introduced by \cite{NeYu83}, the \emph{median-of-means} estimator takes 
a confidence level $\delta\in (0,1)$ and divides the data  into  $V \approx \log \delta^{-1}$ blocks. For each block $k=1,\ldots,V$, one may compute  the empirical mean $ \widehat{\mu}_{k} $ on the variables in the block. The median $\overline{\mu}$ of the $\widehat{\mu}_k$ is the so-called median-of-means estimator. A short analysis of the resulting estimator shows that 
\[
|\overline{\mu}-m_h| \le  c \sqrt{\var{X}}\sqrt{\frac{\log (1/\delta)}{n}}
\]
with probability at least $1-\delta$ for a numerical constant $c$. 
For the details of the proof see \cite{LeOl11}.
When the variance is infinite but a moment of order $1<p\le 2$ exists, the median-of means estimator is still useful, see \cite{BuCeLu13}.
This estimator 
has recently been studied in various contexts. $M$-estimation based on this  technique has been developed by \cite{LeOl11} and generalizations in a multivariate context have been discussed by \cite{HsSa13} and \cite{Mi13}. A similar idea was used in \cite{AlMaSz02}.  An interesting alternative of the median-of-means estimator has been proposed by \cite{Cat12}.

The rest of the paper is organized as follows. In Section \ref{sec:mainres} we introduce a robust estimator of the mean $m_h$ and present performance bounds. In particular, Section \ref{sec:fastrates} deals with the finite variance case. Section \ref{sec:momentp} is dedicated to case when $h$ has a finite $p$-th moment for some $1<p < 2$ for $P$-degenerate kernels. Finally, in Section \ref{sec:clust}, we present an application to clustering problems.

\section{Robust $U$-estimation}
\label{sec:mainres}
In this section we introduce a ``median-of-means''-style estimator of $m_h=\EXP h(X_1,\ldots,X_m)$. 
To define the estimator, one divides the data into $V$ blocks. 
For any $m$-tuple of different blocks, one may compute a (decoupled)
$U$-statistics. Finally, one computes the median of all the obtained
values. The rigorous definition is as follows.

The estimator has a parameter $V\le n$, the number of blocks. 
A partition $\mathcal{B}= (B_1,\ldots,B_V)$ of $\{1,\ldots,n\}$ is called \emph{regular} if for all $K =1,\ldots,V$,
\[
 \left||B_K|-\frac{n}{V}\right| \le 1~.
\]
For any $B_{i_1},\ldots,B_{i_m}$ in $\mathcal{B}$, we set 
\[
I_{B_{i_1},\ldots,B_{i_m}}=\left\{(k_1,\ldots,k_m) : k_j \in B_{i_j}\right\}
\]
and
\[
U_{B_{i_1},\ldots,B_{i_m}}(h)=\frac{1}{|B_{i_1}|\cdots|B_{i_m}|}\sum_{(k_1,\ldots,k_m) \in I_{B_{i_1},\ldots,B_{i_m}}} h(X_{k_1},\ldots,X_{k_m})~.
\]
For any integer $N$ and any vector $(a_1,\ldots,a_N) \in\R^N$, we define the median $\text{Med}(a_1,\ldots,a_N)$ as any number $b$ such that
\[
\big|\{i \le N\ :\ a_i \le b\}\big|\ge \frac{N}{2}\quad \text{and} \quad \big|\{i \le N\ :\ a_i \ge b\}\big|\ge \frac{N}{2}~.
\]
Finally, we define the robust estimator:
\begin{equation}
\label{eq:defrobust}
\overline{U}_{\mathcal{B}}(h) = \text{Med}\{U_{B_{i_1},\ldots,B_{i_m}}(h) : i_j \in\{1,\ldots,V\},1\le i_1 <\ldots<i_m\le V\}~.
\end{equation}

Note that, mostly in order to simplify notation, we only take those 
values of $U_{B_{i_1},\ldots,B_{i_m}}(h)$ into account that correspond to 
distinct indices $i_1 <\cdots<i_m$. Thus, each $U_{B_{i_1},\ldots,B_{i_m}}(h)$
is a so-called decoupled $U$-statistics (see the Appendix for the definition).
One may incorporate all $m$-tuples (not necessarily with distinct
indices) in the computation of the median. However, this has a minor 
effect on the performance. Similar bounds may be proven though with a more complicated
notation. 

A simpler alternative is obtained by taking only ``diagonal'' 
blocks into account. More precisely, let $U_{B_i}(h)$ be the $U$-statistics 
calculated using  the variables in block $B_i$ (as defined in \eqref{eq:defUstat}). One may simply calculate the median of the $V$ different $U$-statistics $U_{B_i}(h)$. This version is easy to analyze because
$\big|\{i \le V\ :U_{B_i}(h) \ge b\}\big|$ 
is a sum of independent random variables. 
However, this simple version is wasteful in the sense that only a small 
fraction of possible $m$-tuples are taken into account. 

In the next two sections we analyze the performance of the 
estimator $\overline{U}_{\mathcal{B}}(h)$.

\subsection{Exponential inequalities for $P$-degenerate kernels with finite variance.}
\label{sec:fastrates}

Next we present a performance bound of the estimator
$\overline{U}_{\mathcal{B}}(h)$ in the case when $\sigma^2$ is
finite. The somewhat more complicated case of infinite second
moment is treated in Section \ref{sec:momentp}.

\begin{thm}
\label{thm:maindege}
Let $X_1,\ldots,X_n$ be i.i.d.\ random variables taking values in $\X$. Let 
$h:\X^m\mapsto \R$ be a symmetric kernel that is $P$-degenerate of order $q-1$. 
Assume $\var{ h(X_1,\ldots,X_m)} =\sigma^2 < \infty $. 
Let $\delta \in (0,\frac{1}{2})$ be such that $ \ceil{\log (1/\delta)} \le \frac{n}{64m}$. Let $\mathcal{B}$ be a regular partition of $\{1,\ldots,n\}$ with $|\mathcal{B}|=32m\ceil{\log (1/\delta)}$. Then, with probability at least $1-2\delta$, we have
\begin{equation}
\left|\overline{U}_{\mathcal{B}}(h) -m_h\right| \le K_m\sigma\left(\frac{\ceil{\log (1/\delta)}}{n}\right)^{q/2}~,
\end{equation}
where $K_m =2^{\frac{7}{2}m+1}m^{\frac{m}{2}}$.
\end{thm}

When $q=m$, the kernel $h-m_h$ is $P$-canonical and the rate of convergence is then given by $(\log \delta^{-1}/n)^{m/2}$.
Thus, the new estimator has a performance similar to standard $U$-statistics
as in (\ref{eq:existresult}) and (\ref{eq:existresult2}) but without 
the boundedness assumption for the kernel.
It is important to note that a disadvantage of the estimator $ \overline{U}_{\mathcal{B}}(h)$ is that it
depends on the confidence level $\delta$ (through the number of blocks). For different
confidence levels, different estimators are used.

Because of its importance in applications, we spell out the special case when $m=q=2$. In Section \ref{sec:clust} we use this result in an example of cluster analysis. 
\begin{cor}
\label{cor:m=2}
Let $\delta \in (0,1/2)$. Let $h : \X^2 \mapsto \R$ be a $P$-canonical kernel with $\sigma^2=\var{h(X_1,X_2)}$ and let $n \ge 128(1+\log(1/\delta))$. Then, with probability at least $1-2\delta$, 
\begin{equation}
|\overline{U}_{\mathcal{B}}(h) -m_h| \le 512\sigma\frac{1+\log(1/\delta)}{n}~.
\end{equation}
\end{cor}

In the proof of Theorem \ref{thm:maindege} we need the notion of \emph{Hoeffding decomposition} \citep{Ho48} of $U$-statistics.
For probability measures $P_1, \ldots, P_m$, define $P_1\times \cdots \times P_m h = \int h\ d(P_1,\ldots,P_m)$.
For a symmetric kernel $h : \X^{m} \mapsto \R$ the \emph{Hoeffding projections} are defined, for $0 \le k \le m$ and $x_1,\ldots,x_k \in \X$, as
\[
\pi_k h (x_1,\ldots,x_k) := (\delta_{x_1}-P)\times \cdots \times (\delta_{x_k}-P)\times P^{m-k} h
\]
where $ \delta_{x} $ denotes the Dirac measure at the point $x$.
Observe that $\pi_0 h=P^m h$ and for $k>0$, $\pi_k h$ is a $P$-canonical kernel. $h$ can be decomposed as
\begin{equation}
\label{eq:decomph}
h (x_1,\ldots,x_m) = \sum_{k=0}^m\sum_{1\le i_1< \ldots < i_k \le m} \pi_k h(x_{i_1},\ldots,x_{i_k})~.
\end{equation}
If $h$ is assumed to be square-integrable (i.e., $P^m h^2 < \infty$), the terms in (\ref{eq:decomph})  are orthogonal. If $h$ is degenerate of order $q-1$, then for any $1\le k \le q-1$, $\pi_k h =0$.

\begin{proof}[Proof of Theorem \ref{thm:maindege}]
We begin with a ``weak'' concentration result on each $U_{B_{i_1},\ldots,B_{i_m}}(h)$. Let $B_{i_1},\ldots,B_{i_m}$ be elements of $\mathcal{B}$. For any $ B \in \B $, we have $\frac{n}{2|\B|} \le |B|\le \frac{2n}{|\B|} $. We denote by $ \mathbf{k}=(k_1,\ldots,k_m) $ an element of $I_{B_{i_1},\ldots,B_{i_m}}$. We have, by the above-mentioned orthogonality property,
\begin{eqnarray*}
\lefteqn{ 
\var{U_{B_{i_1},\ldots,B_{i_m}}(h)}  } \\
& = &   \esp{(U_{B_{i_1},\ldots,B_{i_m}}(h) -P^m h)^2}\\
&= &  \frac{1}{|B_{i_1}|^2\ldots|B_{i_m}|^2}\sum_{\substack{\mathbf{k} \in I_{B_{i_1},\ldots,B_{i_m}} \\ \mathbf{l} \in I_{B_{i_1},\ldots,B_{i_m}} }} \esp{(h(X_{k_1},\ldots,X_{k_m})-P^m h)(h(X_{l_1},\ldots,X_{l_m})-P^m h)}\\
&=& \frac{1}{|B_{i_1}|^2\ldots|B_{i_m}|^2}\sum_{\substack{\mathbf{k} \in I_{B_{i_1},\ldots,B_{i_m}} \\ \mathbf{l} \in I_{B_{i_1},\ldots,B_{i_m}} }} \sum_{s=q}^m \binom{|\mathbf{k}\cap\mathbf{l}|}{s} \esp{\pi_sh(X_1,\ldots,X_s)^2}\quad (\text{by orthogonality})\\
&\le& \frac{1}{|B_{i_1}|^2\ldots|B_{i_m}|^2}\sum_{\mathbf{k} \in I_{B_{i_1},\ldots,B_{i_m}}} \sum_{s=q}^m \sum_{t=0}^m\binom{t}{s} \esp{\pi_sh(X_1,\ldots,X_s)^2} \times \left(\frac{2n}{|\B|}\right)^{m-t}~.
\end{eqnarray*}
The last inequality is obtained by counting, for any fixed $ \mathbf{k} $ and $t$, the number of elements $\mathbf{l}$ such that $ |\mathbf{k}\cap\mathbf{l}|=t $. Thus,
\begin{eqnarray*}
\var{U_{B_{i_1},\ldots,B_{i_m}}(h)}&\le & \frac{1}{|B_{i_1}|\ldots|B_{i_m}|} \sum_{s=q}^m \sum_{t=q}^m\binom{t}{s} \esp{\pi_sh(X_1,\ldots,X_s)^2} \times \left(\frac{2n}{|\B|}\right)^{m-t}\\
&\le & \frac{1}{|B_{i_1}|\ldots|B_{i_m}|} \sum_{s=q}^m \binom{m}{s} \esp{\pi_sh(X_1,\ldots,X_s)^2} \times \sum_{t=q}^m\left(\frac{2n}{|\B|}\right)^{m-t}\\
& \le & \frac{1}{\left(\frac{n}{2|\B|}\right)^m} \sum_{s=q}^m \binom{m}{s} \esp{\pi_sh(X_1,\ldots,X_s)^2} \times 2\left(\frac{2n}{|\B|}\right)^{m-q}\\
& \le & \frac{2^{2m-q+1}|\B|^q}{n^q} \sum_{s=q}^m \binom{m}{s} \esp{\pi_sh(X_1,\ldots,X_s)^2}~.
\end{eqnarray*}
On the other hand, we have, by (\ref{eq:decomph}),
\begin{eqnarray*}
\var{h} &=&  \esp{\left(\sum_{s=q}^m\sum_{1\le i_1< \ldots < i_s \le m} \pi_s h(X_{i_1},\ldots,X_{i_s})\right)^2}\\
&=& \sum_{s=q}^m\sum_{1\le i_1< \ldots < i_s \le m} \esp{\left(\pi_s h(X_{i_1},\ldots,X_{i_s})\right)^2}\\
&=& \sum_{s=q}^m {m \choose s} \esp{\left(\pi_s h(X_1,\ldots,X_s)\right)^2}~.
\end{eqnarray*}
Combining the two displayed equations above,
\begin{eqnarray*}
\var{U_{B_{i_1},\ldots,B_{i_m}}(h)} &\le & \frac{2^{2m-q+1}|\B|^q}{n^q}\sigma^2 \le \frac{2^{2m}|\B|^q}{n^q}\sigma^2~.
\end{eqnarray*}
By Chebyshev's inequality, for all $r \in (0,1)$,
\begin{equation}
\label{eq:chebyshev}
\prob{U_{B_{i_1},\ldots,B_{i_m}}(h)-P^mh > 2^m\sigma\frac{|\B|^{q/2}}{n^{q/2}r^{1/2}}  } \le r~.
\end{equation}
We set $x=2^m\sigma\frac{|\B|^{q/2}}{n^{q/2}r^{1/2}} $, and 
\[
N_x = \left|\left\{(i_1,\ldots,i_m) \in\{1,\ldots,V\}^m : 1\le i_1 <\ldots<i_m\le |B|,\ U_{B_{i_1},\ldots,B_{i_m}}(h)-P^mh >x\right\}\right|~.
\]
The random variable $\frac{1}{\binom{|\B|}{m}} N_x$ is a $U$-statistics of order $m$ with the symmetric kernel $g:(i_1,\ldots,i_m)\mapsto\IND{U_{B_{i_1},\ldots,B_{i_m}}(h)-P^mh >x}$. Thus, Hoeffding's inequality for centered $U$-statistics 
(\ref{eq:basicconcen})
 gives 
\begin{equation}
\label{eq:stobound}
\prob{N_x - \EXP N_x \ge t\binom{|\B|}{m} } \le \exp \left(-\frac{|\B|t^2}{2m}\right)~.
\end{equation}
By (\ref{eq:chebyshev}) we have $ \EXP N_x\le \binom{|\B|}{m}r $. Taking $t=r=\frac{1}{4}$ in \eqref{eq:stobound}, by the definition of the median, we have
\begin{eqnarray*}
\prob{\overline{U}_{\mathcal{B}}(h)-P^m(h) >x } &\le & \prob{N_x \ge \frac{\binom{|\B|}{m} }{2}}\\
&\le & \exp \left(-\frac{|\B|}{32m}\right)~.
\end{eqnarray*}
Since $|\mathcal{B}| \ge 32m\log (\delta^{-1})$, with probability at least $1-\delta$, we have 
\[
\overline{U}_{\mathcal{B}}(h)-P^mh \le K_m\sigma\left(\frac{\ceil{\log \delta^{-1}}}{n}\right)^{q/2}
\]
with $K_m =2^{\frac{7}{2}m+1}m^{\frac{m}{2}}$. The upper bound for the lower tail holds by the same argument. 
\end{proof}

%
%
\subsection{Bounded moment of order $p$ with $1<p\le 2$}
\label{sec:momentp}
In this section, we weaken the assumption of finite variance and only assume the existence of a centered moment of order  $p$ for some $1<p\le 2$. 
The outline of the argument is similar as in the case of finite variance.
First we obtain a ``weak'' concentration inequality for the $U$-statistics
is each block and then use the property of the median to boost the weak inequality. While for the case of finite variance weak concentration could be 
proved by a direct calculation of the variance, here we need the 
randomization inequalities for convex functions of $U$-statistics 
established by \cite{del92} and \cite{ArGi93}. 
Note that, here, a $P$-canonical technical assumption is needed.

\begin{theorem}
\label{thm:momentp}
Let $h$ be a symmetric kernel of order $m$ such that $h-m_h$ is $P$-canonical.
Assume that $M_p:= \esp{\big|h(X_1,\ldots,X_m)-m_h\big|^p}^{1/p} < \infty $ for
some  $1<p\le 2$. Let $\delta \in (0,\frac{1}{2}) $ be such that $ \ceil{\log (\delta^{-1})} \le \frac{n}{64m}$. Let $\mathcal{B}$ be a regular partition of $\{1,\ldots,n\}$ with $|\mathcal{B}|=32m\ceil{\log (\delta^{-1})}$. Then, with probability at least $1-2\delta$, we have
\begin{equation}
\label{eq:thmpmoment}
|\overline{U}_{\mathcal{B}}(h) -m_h| \le K_m M_p\left(\frac{\ceil{\log (\delta^{-1})}}{n}\right)^{m(p-1)/p}
\end{equation}
where $K_m = 2^{4m+1}m^{\frac{m}{2}}$.
\end{theorem}
\begin{proof}
Define the centered version of $h$ by $g(x_1,\ldots,x_m):=h(x_1,\ldots,x_m)-m_h$. Let $\varepsilon_1,\ldots,\varepsilon_n$ be i.i.d. Rademacher random variables (i.e., $ \prob{\varepsilon_1=-1}=\prob{\varepsilon_1=1}=1/2 $) independent of $X_1,\ldots,X_n$. By the randomization inequalities (see Theorem 3.5.3 in \cite{PeGi99} and also Theorem \ref{thm:randomize} in the Appendix), we have
\begin{eqnarray}
\lefteqn{
\esp{\left|\sum_{(k_1,\ldots,k_m) \in I_{B_{i_1},\ldots,B_{i_m}}} g(X_{k_1},\ldots,X_{k_m})\right|^p} }\nonumber \\ 
&\le & 2^{mp} \EXP_X\espi{\varepsilon}{\left|\sum_{(k_1,\ldots,k_m) \in I_{B_{i_1},\ldots,B_{i_m}}} \varepsilon_{k_1}\ldots\varepsilon_{k_m} g(X_{k_1},\ldots,X_{k_m})\right|^p}\label{ineq:randomisation}\\ 
& \le & 2^{mp} \espi{X}{\left|\espi{\varepsilon}{\left(\sum_{(k_1,\ldots,k_m) \in I_{B_{i_1},\ldots,B_{i_m}}} \varepsilon_{k_1}\ldots\varepsilon_{k_m} g(X_{k_1},\ldots,X_{k_m})\right)^2}\right|^{p/2}}\nonumber\\ 
& = & 2^{mp} \espi{X}{ \left|\sum_{(k_1,\ldots,k_m) \in I_{B_{i_1},\ldots,B_{i_m}}}  g(X_{k_1},\ldots,X_{k_m})^2\right|^{p/2}}\nonumber\\
& \le & 2^{mp}  \sum_{(k_1,\ldots,k_m) \in I_{B_{i_1},\ldots,B_{i_m}}} \EXP |g(X_{k_1},\ldots,X_{k_m})|^p\nonumber\\
&=& 2^{mp} |B_{i_1}|\cdots|B_{i_m}|\EXP |g|^p~. \label{ineq:convex2}
\end{eqnarray}
Thus, we have $\esp{ |U_{B_{i_1},\ldots,B_{i_m}}(h)-m_h|^p} \le 2^{mp} (|B_{i_1}|\ldots|B_{i_m}|)^{1-p}\EXP |g|^p$ and by Markov's inequality,
\begin{equation}
\label{eq:markov}
\prob{U_{B_{i_1},\ldots,B_{i_m}}(h)-m_h>\frac{2^m M_p}{r^{\frac{1}{p}}}\left(\frac{n}{(2|\B|)}\right)^{m\frac{1-p}{p}}} \le r~.
\end{equation}
Another use of \eqref{eq:stobound} with $t=r=\frac{1}{4}$ gives
\[
\overline{U}_{\mathcal{B}}(h)-P^mh \le 2^{4m+1}m^{\frac{m}{2}}M_p\left(\frac{\ceil{\log \delta^{-1}}}{n}\right)^{m\frac{p-1}{p}}~.
\]
\end{proof}

To see why the bound of Theorem \ref{thm:momentp} gives essentially 
the right order of magnitude, consider again the example described in
the introduction, when $m=2$, $ h(X_1,X_2)=X_1X_2$, and the $X_i$ have
an $\alpha$-\emph{stable} law $S(\gamma,\alpha)$ for some
$\gamma >0$ and $1< \alpha \le 2$. Note that an $\alpha$-stable random variable
has finite moments up to (but not including) $\alpha$ and therefore
we may take any $p=\alpha-\epsilon$ for any $\epsilon \in (0,1-\alpha)$.
As we noted it in the introduction, there exists a constant $c$ depending
on $\alpha$ and $\gamma$ only such that for all $1\le i_1 < i_2 \le V$,
\[
\prob{\left|U_{B_{i_1},B_{i_2}}(h)-m_h\right| \ge c \left(\frac{n}{|\B|}\right)^{2/\alpha-2}} \ge 2/3~,
\]
and therefore (\ref{eq:markov}) is essentially the best rate one can
hope for. 

\section{Cluster analysis with $U$-statistics}
\label{sec:clust}
In this section we illustrate the use of the proposed mean estimator
in a clustering problem when the presence of possibly heavy-tailed 
data requires robust techniques.

We consider the general statistical framework defined by \cite{Cle14}, 
described as follows:
Let  $X,X'$ be i.i.d.\ random variables taking values in $\X$ where typically but not necessarily, $\X$ is a subset of $\R^d$). 
For a partition $\Pa$ of $\X$ into $K$ disjoint sets--the so-called ``cells''--,
define $\Phi_{\Pa}(x,x')=\sum_{\mathcal{C} \in \Pa} \IND{(x,x') \in \mathcal{C}^2}$ the $\{0,1\}$-valued function that indicates whether two elements $x$ and $x'$ belong to the same cell $\mathcal{C}$. Given a \emph{dissimilarity measure}
$D : \X^2 \rightarrow \R_+^*$, the clustering task consists in finding 
a partition of $\X$ minimizing the \emph{clustering risk}
\[
W(\Pa) = \EXP \left[D(X,X') \Phi_{\Pa} (X,X') \right]~.
\]
Let $\Pi_K$ be a  finite class of partitions $\Pa$ of $\X$ into $K$ cells and
define $W^*= \min_{\Pa \in\Pi_K} W(\Pa)$.  

Given $X_1,\ldots,X_n$ be i.i.d.\ random variables distributed as $X$,
the goal is to find a partition $\Pa \in\Pi_K$ with risk as close to 
$W^*$ as possible. A natural idea--and this is the approach of \cite{Cle14}--is
to estimate $W(\Pa)$ by the $U$-statistics
\[
\widehat{W}_n(\Pa)=\frac{2}{n(n-1)} \sum_{1 \le i < j \le n} D(X_i,X_j) \Phi_{\Pa} (X_i,X_j)
\]
and choose a partition minimizing the empirical clustering risk $\widehat{W}_n(\Pa)$. \cite{Cle14} uses the theory of $U$-processes to analyze the
performance of such minimizers of $U$-statistics. However, in order to 
control uniform deviations of the form 
$\sup_{\Pa\in \Pi_K}|\widehat{W}_n(\Pa)-W(\Pa)|$, exponential concentration
inequalities are needed for $U$-statistics. This  restricts one to consider
bounded dissimilarity measures $D(X,X')$.  When $D(X,X')$ may have
a heavy tail, we propose to replace $U$-statistics by the median-of-means
estimators of $W(\Pa)$ introduced in this paper.

Let $\B$ be a regular partition of $ \{1,\ldots,n\} $ and define the 
median-of-means estimator $\overline{W}_{\B}(\Pa)$ of $W(\Pa)$ as in (\ref{eq:defrobust}). Then Theorem \ref{thm:maindege} applies and we have the following
simple corollary.
\begin{cor}
\label{cor:cluster1}
Let $\Pi_K$ be a class of partitions  of cardinality $|\Pi_K|=N$. 
Assume that $\sigma^2:=\esp{D(X_1,X_2)^2}< \infty$. Let $\delta \in (0,1/2)$ be such that $n \ge 128\ceil{\log (N/\delta)}$. Let $\B$ be a regular partition of $ \{1,\ldots,n\} $ with $|\B|=64\ceil{\log (N/\delta)}$.  Then there exists a constant $C$ such that, with probability at least $1-2\delta$, 
\begin{equation}
\label{eq:Clust1}
\sup_{\Pa \in \Pi_K}|\overline{W}_{\mathcal{B}}(\Pa) -W(\Pa)| \le C\sigma \left(\frac{\ceil{\log (N/\delta)}}{n}\right)^{1/2}~.
\end{equation}
\end{cor}
\begin{proof}
Since $\Phi_{\Pa} (x,x')$ is bounded by 1, $\var{D(X_1,X_2) \Phi_{\Pa} (X_1,X_2)} \le \esp{D(X_1,X_2)^2}$. For a fixed $\Pa \in \Pi_K$, Theorem \ref{thm:maindege} applies with $m=2$ and $q=1$. The inequality follows from the union bound.
\end{proof}

Once uniform deviations of $\overline{W}_{\mathcal{B}}(\Pa)$ from its expected
value are controlled, it is a routine exercise to derive performance bounds
for clustering based on minimizing $\overline{W}_{\mathcal{B}}(\Pa)$ over
$\Pa\in \Pi_K$.

Let $\widehat{\Pa} = \argmin_{\Pa \in\Pi_K} \overline{W}_{\B}(\Pa)$ denote the
empirical minimizer. (In case of multiple minimizers, one may select one arbitrarily.) Now for any $\Pa_0 \in \Pi_K$,
\begin{eqnarray*}
W(\widehat{\Pa})-W^* &=& W(\widehat{\Pa})-\overline{W}_{\B}(\widehat{\Pa})+\overline{W}_{\B}(\widehat{\Pa})-W^* \\ 
   &\le & W(\widehat{\Pa})-\overline{W}_{\B}(\widehat{\Pa})+\overline{W}_{\B}(\Pa_0)-W(\Pa_0)+W(\Pa_0)-W^*\\ 
   & \le & 2 \sup_{\Pa \in \Pi_K}|\overline{W}_{\mathcal{B}}(\Pa) -W(\Pa)| + W(\Pa_0)-W^*~.
\end{eqnarray*}
Taking the infimum over $\Pi_K$, 
\begin{equation}
\label{eq:split}
W(\widehat{\Pa})-W^* \le 2 \sup_{\Pa \in \Pi_K}|\overline{W}_{\mathcal{B}}(\Pa) -W(\Pa)|~.
\end{equation}
Finally, (\ref{eq:Clust1}) implies that
\[
W(\widehat{\Pa})-W^* \le 2C\sigma \left(\frac{1+\log (N/\delta)}{n}\right)^{1/2}~.
\]
This result is to be compared with Theorem 2 of \cite{Cle14}. 
Our result holds under the only  assumption that $D(X,X')$ has a finite second moment. (This may be weakened to assuming the existence
of a finite $p$-th moment for some $1<p\le 2$ by using Theorem \ref{thm:momentp}). 
On the other hand, our result holds only for a finite class of partitions
while \cite{Cle14} uses the theory of $U$-processes to obtain
more sophisticated bounds for uniform deviations over possibly infinite
classes of partitions. It remains a challenge to develop a theory
to control processes of median-of-means estimators--in the style of \cite{ArGi93}--and not having to resort
to the use of simple union bounds.
 
In the rest of this section we show that, under certain ``low-noise'' assumptions,
analogous to the ones introduced by \cite{Tsy99} in the context of classification, to obtain faster rates of convergence. 
In this part we need bounds for $P$-canonical kernels and use the full power of Corollary \ref{cor:m=2}.
Similar arguments for the
study of minimizing $U$-statistics appear in \cite{CleVaLu08}, \cite{Cle14}.

We assume the following conditions, also considered by \cite{Cle14}:
\begin{enumerate}
\item There exists $\Pa^*$ such that $ W(\Pa^*)=W^* $
\item There exist $\alpha \in [0,1]$ and $\kappa < \infty$ such that for all 
$\Pa \in \Pi_K$ and for all $x \in \X$,
\[
\prob{\Phi_{\Pa}(x,X) \neq \Phi_{\Pa^*}(x,X)} \le \kappa (W(\Pa)-W^*)^{\alpha}~.
\]
\end{enumerate}
Note  that $\alpha \le 2$ since by the Cauchy-Schwarz inequality,
\[
W(\Pa)-W^* \le \esp{D(X_1,X_2)^2}^{1/2} \prob{\Phi_{\Pa}(X_1,X_2) \neq \Phi_{\Pa^*}(X_1,X_2)}^{1/2}~.
\]

\begin{cor}
\label{cor:Cluster2}
Assume the conditions above and that $\sigma^2:=\esp{D(X_1,X_2)^2}< \infty$. Let $\delta \in (0,1/2)$ be such that $n \ge 128\ceil{\log (N/\delta)}$. Let $\B$ be a regular partition of $ \{1,\ldots,n\} $ with $|\B|=64\ceil{\log (N/\delta)}$. Then there exists a constant $C$ such that, with probability at least $1-2\delta$, 
\begin{equation}
\label{eq:Clust2}
W(\widehat{\Pa})-W^* \le C\sigma^{2/(2-\alpha)}\left( \frac{\ceil{\log (N/\delta)}}{n}\right)^{1/(2-\alpha)}~.
\end{equation}
\end{cor}
 The proof Corollary \ref{cor:Cluster2} is postponed to the Appendix.

\section{Appendix}

\subsection{Decoupling and randomization}  

Here we summarize some of the key tools for analyzing $U$-statistics that
we use in the paper. For an excellent exposition we refer to \cite{PeGi99}.

Let $\{X_i\}$ be i.i.d.\ random variables taking values in $\X$ and
let $\{X_i^k\},\ k=1,\ldots, m$, be sequences of independent copies.
Let $\Phi$ be a non-negative function. 
As a corollary of Theorem 3.1.1 in \cite{PeGi99} we have the following:
\begin{theorem}
\label{thm:decoupling}
Let $h : \X^m \rightarrow \R$ be a measurable function with $\EXP |h(X_1,\ldots,X_m)| <\infty$. Let $\Phi: [0, \infty ) \rightarrow [0, \infty)$ be a convex nondecreasing function such that $\EXP \Phi \left( |h(X_1,\ldots,X_m)|\right) < \infty$. Then
\[
\EXP \Phi \left( \left|\sum_{ I_n^m}h(X_{i_1},\ldots,X_{i_m})\right|\right) \le \EXP \Phi \left( C_m \left|\sum_{ I_n^m} h(X_{i_1}^1,\ldots,X_{i_m}^m)\right|\right)
\]
where $C_m=2^m(m^m-1)((m-1)^{m-1}-1)\times\cdots\times 3$.
Moreover, if the kernel $h$ is symmetric, then,
\[
\EXP \Phi \left( c_m \left|\sum_{ I_n^m} h(X_{i_1}^1,\ldots,X_{i_m}^m)\right|\right) \le \EXP \Phi \left( \left|\sum_{ I_n^m}h(X_{i_1},\ldots,X_{i_m})\right|\right)
\]
where $c_m=1/(2^{2m-2}(m-1)!)$.
\end{theorem}
An equivalent result for tail probabilities of $U$-statistics is the following (see Theorem 3.4.1 in \cite{PeGi99}):
\begin{theorem}
\label{thm:decoutail}
Under the same hypotheses as Theorem \ref{thm:decoupling}, there exists a constant $C_m$ depending  on $m$ only such that, for all $t>0$,
\[
\prob{ \left|\sum_{ I_n^m}h(X_{i_1},\ldots,X_{i_m})\right|>t } \le C_m \prob{C_m\left|\sum_{ I_n^m} h(X_{i_1}^1,\ldots,X_{i_m}^m)\right|>t}~.
\]
If moreover, the kernel $h$ is symmetric then there exists a constant $c_m$ depending on $m$ only such that, for all $t>0$,
\[
c_m \prob{c_m\left|\sum_{ I_n^m} h(X_{i_1}^1,\ldots,X_{i_m}^m)\right|>t} \le \prob{ \left|\sum_{ I_n^m}h(X_{i_1},\ldots,X_{i_m})\right|>t }~.
\]
\end{theorem}
The next Theorem is a direct corollary of Theorem 3.5.3 in \cite{PeGi99}. 
\begin{theorem}
\label{thm:randomize}
Let $1<p\le 2$. Let $(\varepsilon_i)_{i \le n} $ be i.i.d Rademacher random variables independent of the $(X_i)_{i \le n}$. Let $h : \X \rightarrow \R$ be a $P$-degenerate measurable function such that $\EXP  \left( |h(X_1,\ldots,X_m)|^p\right) < \infty$. Then
\begin{eqnarray*}
c_m\EXP\ \Big|\sum_{ I_n^m}\varepsilon_{i_1}\ldots\varepsilon_{i_m} h(X_{i_1},\ldots,X_{i_m})\Big|^p & \le & \EXP\ \Big|\sum_{ I_n^m} h(X_{i_1},\ldots,X_{i_m})\Big|^p
\\
& \le &
C_m \EXP\ \Big|\sum_{ I_n^m}\varepsilon_{i_1}\ldots\varepsilon_{i_m} h(X_{i_1},\ldots,X_{i_m})\Big|^p~,
\end{eqnarray*}
where $C_m=2^{mp}$ and $c_m=2^{-mp}$.
\end{theorem}
The same conclusion holds for decoupled $U$-statistics.

\subsection{$\alpha$-stable distributions}

\begin{prop}
\label{prop:stable}
Let $\alpha \in (0,2) $. Let $ X_1,\ldots,X_n $ be i.i.d. random variables of law $S(\gamma,\alpha)$. Let $f_{\gamma,\alpha} :x \mapsto \R $ be the density function of $X_1$. Let $S_n =\sum_{1 \le i\le n} X_i$. Then 
\begin{itemize}
\item[(i)] $f_{\gamma,\alpha}(x)$ is an even function.
\item[(ii)] $ f_{\gamma,\alpha}(x) \underset{x \to +\infty}{\sim}\alpha\gamma^{\alpha}c_{\alpha}x^{-\alpha-1} $ with $c_\alpha = \sin (\frac{\pi \alpha}{2})\Gamma(\alpha)/\pi$.
\item[(iii)] $\esp{X_1^p}$ is finite for any $p < \alpha$ and is infinite whenever $ p \ge \alpha$. 
\item[(iv)] $S_n$ has a $\alpha$-stable law $S(\gamma n^{1/\alpha},\alpha)$.
\end{itemize}
\end{prop}
\begin{proof}
(i) and (iv) follow directly from the definition. (ii) is proved in the introduction of  \cite{Zo86}. (iii) is a consequence of (ii).
\end{proof}

\subsection{Proof of Corollary \ref{cor:Cluster2}}
Define $\Lambda_n(\Pa)=\widehat{W}_n(\Pa) -W^* $, the $U$-statistics based on the sample $X_1,\ldots,X_n$, with symmetric kernel
\[
h_{\Pa}(x,x')=D(x,x')\left(\Phi_{\Pa} (x,x')-\Phi_{\Pa^*} (x,x')\right)~.
\]
We denote by $\Lambda(\Pa)=W(\Pa)-W^*$ the expected value of $\Lambda_n(\Pa)$.
The main argument in the following analysis is based on the Hoeffding decomposition. For all partitions $\Pa$, 
\[
\Lambda_n(\Pa)-\Lambda(\Pa) = 2L_n(\Pa)+M_n(\Pa) 
\]
for $L_n(\Pa)=\frac{1}{n}\sum_{i\le n} h^{(1)}(X_i)$ with $h^{(1)}(x)=\esp{h_{\Pa}(X,x)}-\Lambda(\Pa)$ and $M_n(\Pa)$ the $U$-statistics based on the canonical kernel given by $ h^{(2)}(x,x')=h_{\Pa}(x,x')-h^{(1)}(x)-h^{(1)}(x')-\Lambda(\Pa) $. Let $\mathcal{B}$ be a regular partition of $\{1,\ldots,n\}$. For any $B \in \mathcal{B}$, $\Lambda_{B}(\Pa)$ is the $U$-statistics on the kernel $h_{\Pa}$ restricted to the set $B$ and $ \overline{\Lambda}_{B}(\Pa) $ is the median of the sequence $\left(\Lambda_{B}(\Pa)\right)_{B\in \B}$. We define similarly $L_{B}(\Pa)$ and $M_{B}(\Pa)$ on the variables $(X_i)_{i \in B}$. For any $B\in \B$,
\begin{eqnarray*}
\var{\Lambda_B(\Pa)}&=&4\var{L_B(\Pa)}+\var{M_B(\Pa)}\\
&=&\frac{4}{|B|}\var{h^{(1)}(X)}+\frac{2}{|B|(|B|-1)}\var{h^{(2)}(X_1,X_2)}~.
\end{eqnarray*} 
Simple computations show that $\var{h^{(2)}(X_1,X_2)} = 2\var{h^{(1)}(X)}$ and therefore,
\[
\var{\Lambda_B(\Pa)} \le \frac{8}{|B|}\var{h^{(1)}(X)}~.
\]
Moreover,
\begin{eqnarray*}
\var{h^{(1)}(X)} &\le & \espi{X'}{\espi{X}{h_{\Pa}(X,X')}^2}\\
&\le & \espi{X'}{\espi{X}{D(X,X')^2}\espi{X}{\left(\Phi_{\Pa} (X,X')-\Phi_{\Pa^*} (X,X')\right)^2}}\\
&= & \espi{X'}{\espi{X}{D(X,X')^2}\probi{X}{\Phi_{\Pa} (X,X')\neq\Phi_{\Pa^*} (X,X')}}\\
& \le & \sigma^2 \kappa (W(\Pa)-W^*)^{\alpha}
\end{eqnarray*}
where $\EXP_{X}$ (resp. $\EXP_{X'}$) refers to the expectation taken with respect to $X$ (resp. $X'$). Chebyshev's inequality gives, for $r \in (0,1)$,
\[
\prob{\Lambda_B(\Pa) -\Lambda(\Pa) > \sigma (W(\Pa)-W^*)^{\alpha/2}\sqrt{\frac{8\kappa}{r|B|}}} \le r~.
\]
Using again \eqref{eq:stobound} with $r = \frac{1}{4}$, by  $ |B| \ge \frac{n}{128\ceil{\log (N/\delta)}} $, there exists a constant $C$ such that for any $\Pa \in \Pi_K$, with probability at least $1-2\delta/N$,
\[
|\overline{\Lambda}_{\B}(\Pa)-\Lambda(\Pa)| \le C\sigma (W(\Pa)-W^*)^{\alpha/2}\sqrt{\frac{\ceil{\log (N/\delta)}}{n}}~.
\]
This implies by the union bound, that
\[
|\overline{W}_{\B}(\widehat{\Pa})-W(\widehat{\Pa})| \le K\sigma (W(\widehat{\Pa})-W^*)^{\alpha/2}\sqrt{\frac{\ceil{\log (N/\delta)}}{n}}
\]
with probability at least $1-2\delta$. Using  \eqref{eq:split}, we obtain
\[
(W(\widehat{\Pa})-W^*)^{1-\alpha/2} \le 2 K\sigma \sqrt{\frac{\ceil{\log (N/\delta)}}{n}}~,
\]
concluding the proof.

\bibliographystyle{chicago}
\bibliography{biblio}
\end{document}